\documentclass{amsart}
\usepackage{amsfonts, amsmath,  amssymb}
\newtheorem{theorem}{Theorem}
\theoremstyle{plain}
\newtheorem{corollary}{Corollary}
\newtheorem{definition}{Definition}
\newtheorem{lemma}{Lemma}
\theoremstyle{remark}
\newtheorem{remark}{Remark}

\newcommand{\Po}{{\mathcal P}}
\newcommand{\Li}{{\mathcal L}}
\newcommand{\Cl}{{\mathcal C}}

\begin{document}
\title[Cycles in Projective Spaces]{Cycles in Projective Spaces}
\author{Elaina Aceves}
\address{Department of Mathematics\\
California State University, Fresno \\ Fresno, CA 93740.}
\email{ekaceves@mail.fresnostate.edu}
\author{David Heywood}
\address{Department of Mathematics\\
California State University, Fresno \\ Fresno, CA 93740.}
\email{davaudoo@mail.fresnostate.edu}
\author{Ashley Klahr}
\address{Department of Mathematics and Computer Science \\
University of San Diego \\
San Diego, CA 92110} 
\email{aklahr@sandiego.edu}
\author{Oscar Vega}
\address{Department of Mathematics\\
California State University, Fresno \\ Fresno, CA 93740.}
\email{ovega@csufresno.edu}
\date{}

\subjclass[2000]{Primary 05, 51; Secondary 20} 

\thanks{Part of this work was supported by NSF Grant \#DMS-1156273 (Fresno State's Summer REU)} 

\begin{abstract}
We prove that every possible $k$-cycle can be embedded into $PG(n,q)$, for all $n\geq 3$ and $q$ a power of a prime. 
\end{abstract}

\maketitle

\section{Introduction}

In recent years, several articles have been written on the study of embeddings of cycles, and other families of graphs, into finite affine and projective planes.  For example, Lazebnik, Mellinger, and the fourth author proved in \cite{LMV13} that finite planes (both affine and projective) admit embeddings of cycles of all possible lengths. They also found, in \cite{LMV09}, close expressions for the number of $k$-cycles in any projective plane, for $k=3, 4, 5, 6$. This work was later expanded by Voropaev in \cite{V13} to $k= 7, 8, 9, 10$. In addition, the work of the fourth author with Peabody and White in \cite{PVW13}  shows interesting connections between these embeddings and the structure of the finite field coordinatizing the `host' plane.

Not much is known (really, nothing at all for $n>2$)  about embeddings of cycles into $PG(n,q)$. For $n=2$, besides the articles mentioned in the previous paragraph, Singer's work \cite{Singer38} implies the existence of cycles of maximal length in $PG(2,q)$, and Schmeichel \cite{Schm89} proved pancyclicity of $PG(2,q)$ for when $q$ is prime. We refer the reader to \cite{LMV13} for a thorough introduction to embeddings of graphs into finite structures. 

\medskip

We start with a few definitions, conventions, and well-known facts. Any concept we miss to define may be found in \cite{B98}, \cite{Demb}, or \cite{H}.

A graph $G$ is a finite collection of vertices together with edges connecting pairs of vertices. We denote the set of vertices of $G$ by $V(G)$, and the set of edges of $G$ by $E(G)$. A $k$-cycle $C_k$ is a graph on $k$ vertices and $k$ edges such that every vertex has degree two.  We will say that $k$ is the length of $C_k$. 

Throughout this article, $q$ will always be the power of a prime, and $GF(q)$ will denote the field of order $q$. 

\begin{definition}
For $n\geq 3$, we let $V$ to be an $(n+1)$-dimensional vector space over $GF(q)$. The \emph{$n$-dimensional finite projective space}, denoted $PG(n,q)$, is the geometry constructed from $V$ as follows:
\begin{itemize}
\item its points are the $1$-dimensional subspaces of $V$,
\item its lines are the $2$-dimensional subspaces of $V$,
\item in general, its $k$-dimensional projective subspaces are the $(k+1)$-dimensional subspaces of $V$, for all $k=1,2,\cdots , n$.
\item Incidence is given by standard set-theoretical inclusion. 
\end{itemize}
\end{definition}

The construction above is also valid for $n=2$. However, in this case we cannot say that $PG(2,q)$ is \emph{the} projective plane, as there may be several non-isomorphic projective planes constructed on the same set of points (this is not the case for dimension $3$ and above).  From now on we assume $n\geq 3$.\\

\begin{remark}\label{rmkcountingpgnq}
Just like projective planes, $PG(n,q)$ is very symmetric. The  following results are not too hard to prove (and may be found in \cite{H}). 
\begin{itemize}
\item[\textbf{(a)}] $PG(n,q)$ contains 
\[
\frac{q^{n+1}-1}{q-1} \ \text{points, \hspace{.5in}and \hspace{.5in}} \frac{(q^{n+1}-1)(q^{n}-1)}{(q^2-1)(q-1)}\  \text{lines}
\]
Note that there are more lines than points.
\item[\textbf{(b)}] Every line of $PG(n,q)$ contains exactly $q+1$ points.
\item[\textbf{(c)}] Every plane of $PG(n,q)$ is isomorphic to $PG(2,q)$.
\end{itemize}
\end{remark}

\begin{definition}
A collineation of $PG(n,q)$ is a permutation of the points of $PG(n,q)$ that preserves incidence. The group of all collineations of $PG(n,q)$ is denoted $Aut(PG(n,q))$.
\end{definition}

\medskip

Now that we have seen the main two objects we will study, we define the way they will interact with each other.

\begin{definition}
Let $\Po$ denote the set of points in $PG(n,q)$ and $\Li$ denote the set of lines in $PG(n,q)$. We will say that a graph $G=(V,E)$ is embedded (or embeds) into $PG(n,q)$ if there is an injective map $\varphi:V \rightarrow \mathcal{P}$, that, by preserving incidence, induces a injective map $\overline{\varphi}:E \rightarrow \mathcal{L}$. \\
In the case such a function exists we will call it an embedding of $G$ into $PG(n,q)$, and write $G \hookrightarrow PG(n,q)$, or just $\varphi: G \hookrightarrow PG(n,q)$ when the context allows it.
\end{definition}

\begin{remark}
If $\phi : G\hookrightarrow PG(n,q)$, we will identify $G$ with $\phi(G)$, the vertex $v$ with the point $\phi(v)$, and the edge $e$ connecting vertices $v$ and $w$ with the \emph{whole} line $\phi(v)\phi(w)$. Hence, when the context allows it, we will call points in $PG(n,q)$ vertices, and lines in $PG(n,q)$ edges.
\end{remark}

We can now start studying embeddings of cycles into $PG(n,q)$. Our ultimate goal, in this article, is to prove the following result.

\begin{theorem} \label{thmMAIN1}
$PG(n,q)$ is pancyclic. That is, $k$-cycles embed in $PG(n,q)$, for all $3\leq k \leq q^{n}+q^{n-1}+\cdots + q+1$. 
\end{theorem}

Our approach to embedding cycles into projective spaces is of a recursive nature.  It is based on the results in \cite{LMV13} and the few results in the following short section.

\section{Cycles in $PG(2,q)$}

Let $\pi$ be a projective plane in $PG(3,q)$,  isomorphic to $PG(2,q)$, with  `origin' $O$ and line at infinity $\ell_{\infty}$. Following the notation in \cite{LMV13}, we will denote the lines through $O$ by $l_i$, for $1\leq i \leq q+1$, and the points $l_i \cap \ell_{\infty}$ by $(i)$.

\begin{remark}\label{rem1}
The proof of Theorem 1 in \cite{LMV13} not only shows that a $k$-cycle $C_k$ may be embedded in an affine plane of order $q$, for all $3\leq k \leq q^2$. In fact, a close analysis of this proof also yields that $C_k$ may be embedded in one of the following two ways:
\begin{itemize}
\item[\textbf{(a)}] $O$ is a vertex of $C_k$, and there is at least one line through $O$ that is not an edge of $C_k$.
\item[\textbf{(b)}] $O$ is not a vertex of $C_k$, $k=t(q+1)$, for some $1\leq t \leq q-1$, and no line through $O$ is an edge of $C_k$.
\end{itemize}
\end{remark}

Next is the result we really need for later on.

\begin{corollary}\label{cor1}
For every $3\leq k \leq q^2+2$, a path on $k$ vertices can be embedded into $\pi$ in such a way that its endpoints lay on $\ell_{\infty}$ and no other point on $\ell_{\infty}$ is a vertex of this path.
\end{corollary}

\begin{proof}
The cases $k=3, 4$ are trivial, so let $5\leq k\leq q^2+2$ and consider a $(k-2)$-cycle $\Cl$ with the characteristics mentioned in Remark \ref{rem1}. \\
If $O$ is a vertex of $\Cl$. Let $l_i$ and $l_j$ be the two edges of $\Cl$ through $O$, and let $l_t$ be a line through $O$ that is not an edge of $\Cl$. Now delete the edge $l_i$ from $\Cl$ to get a path $\mathcal{P}$ on $k-2$ vertices. Then the path
\[
(i) \xrightarrow{l_{i}} \mathcal{P}   \xrightarrow{l_{t}}  (t)     
\]
has length $k$ and the desired properties.\\
If $O$ is not a vertex of $\Cl$ then we consider three `consecutive' vertices in $\Cl$: $P_i, P_t$, and $P_j$, which are on the lines through $O$  labeled by $l_i, l_t$, and $l_j$, respectively. We delete the edges $P_iP_t$ and $P_tP_j$, and the vertex $P_t$,  from $\Cl$ to get a path $\mathcal{P}$ on $k-3$ vertices with endpoints $P_i$ and $P_j$. Then the path
\[
(i) \xrightarrow{l_{i}} \mathcal{P}   \xrightarrow{l_{j}}  O       \xrightarrow{l_{t}}  (t)
\]
has length $k$ and the desired properties.
\end{proof}

We now have enough to address our main problem.

\section{Cycles in $PG(n,q)$}

In order to construct cycles in $PG(n,q)$ we will `glue' paths on certain projective subspaces of $PG(n,q)$, that have been arranged in a `nice' way.  For this strategy to work, we need the following remark.

\begin{remark}[See \cite{H}]\label{rem5}
Given a projective subspace $\Sigma \cong PG(n-1,q)$ of $PG(n+1,q)$, there are exactly $q+1$ $n$-dimensional projective subspaces (all isomorphic to $PG(n,q)$) of $PG(n+1,q)$ containing $\Sigma$. These $PG(n,q)$s partition the points of $PG(n+1,q) \setminus \Sigma$. \\
The collineation group of $PG(n+1,q)$ is transitive on the set of its $n$-dimensional projective subspaces.
\end{remark}

The following technical lemma is `folklore', and easy to prove using elementary linear algebra. We include it here for future reference.

\begin{lemma}\label{lemPGn-1inPGn}
Let $\Pi = PG(n,q)$ and let 
\[
\Sigma = \{ <(x_1, \ldots , x_{n+1})> \ \in \Pi ; \ x_{n+1}=0   \}
\] 
which is isomorphic to  $PG(n-1,q)$. Let 
\[
Aut(\Sigma) =\{ \tau \in Aut(\Pi); \ \tau(\Sigma) = \Sigma \}
\]
and
\[
 Aut_{\Sigma} =\{ \tau \in Aut(\Pi); \ \tau(P) = P, \ \text{for all} \ P\in \Sigma \}
 \]
Then, 
\begin{itemize}
\item[\textbf{(a)}]  $Aut(\Sigma)$ acts transitively on the set of lines of $\Sigma$. Moreover, the stabilizer of a line acts doubly-transitively on the points of such line. 
\item[\textbf{(b)}]   Both $Aut(\Sigma)$ and $Aut_{\Sigma}$ act on the points of $\Pi \setminus \Sigma$. Moreover, The orbit of $P\in \Pi \setminus \Sigma$ under $Aut_{\Sigma}$ consists of at least $q-1$ elements.
\end{itemize}
\end{lemma}

\begin{remark}
Lemma \ref{lemPGn-1inPGn} holds for any $\Sigma$, isomorphic to $PG(n-1,q)$, contained in any projective space $\Pi$ that is isomorphic to $PG(n,q)$.
\end{remark}

\begin{lemma}\label{lemshortedcycles}
Let $PG(n-1,q) \cong \Sigma \subseteq PG(n,q)$. For every $3\leq k \leq q^{n}+2$, we can construct a $k$-cycle on $PG(n,q)$ with exactly two points of $\Sigma$ as vertices, and exactly one line of $\Sigma$ as edge. 
\end{lemma}

\begin{proof}
We will proceed by induction on $n$. The result is true for $n=2$ because of Corollary \ref{cor1}.  Hence, we assume it is true for $n\leq d$. 

For $n=d+1\geq 3$, we use Remark \ref{rem5} to  `partition' $PG(d+1,q)$ into $q+1$ copies of $PG(d, q)$, labelled $\Pi_1, \ldots , \Pi_{q+1}$, that share $\Sigma \cong PG(d-1,q)$. 

If $3\leq k \leq q^{d}+2$ then the induction hypothesis gives us a $k$-cycle in $\Pi_1$ with exactly two points of $\Sigma$ as vertices, and exactly one line of $\Sigma$ as edge.  Since $\Pi_1\cap \Pi_2 = \Sigma$, this cycle can be seen as a $k$-cycle in $PG(n,q)$ with exactly two points of $\Pi_2$ as vertices, and exactly one line of $\Pi_2$ as edge. \\

Let $q^{d}+3\leq k \leq q^{d+1}+2$. We write $k=\alpha q^d +\beta$, where $0\leq \beta < q^d$, and $1 \leq \alpha \leq q$. \\

\noindent \textbf{Case 1. $\beta=0,1,2$:} First notice that $\beta=0,1,2$ forces $\alpha>1$. \\
For every $1\leq i \leq \alpha -1$ we use the induction hypothesis to get a cycle $\Cl_i$ in $\Pi_i$ of length $q^d+2$ with exactly two points of $\Sigma$ as vertices, and exactly one line of $\Sigma$ as edge. Similarly, we get a cycle $\Cl_{\alpha}$ in $\Pi_{\alpha}$ of length $q^d+\beta$   with exactly two points of $\Sigma$ as vertices, and exactly one line of $\Sigma$ as edge. We now use Lemma  \ref{lemPGn-1inPGn}\textbf{(a)} to get all these cycles to share the edge in $\Sigma$ and the two vertices in $\Sigma$. We will call this common edge $\ell$ and the two common vertices $P$ and $Q$. 

If $\alpha = 2$, then we delete $\ell$ from $\Cl_1$ and from $\Cl_2$, and notice that the union of $\Cl_1\setminus \ell$ and $\Cl_2\setminus \ell$ yields a cycle with the desired length and properties. 

If $\alpha >2$ then, for every $1< i < \alpha$, we delete the three edges incident with $P$ and $Q$ in $\Cl_i$. The path left will be denoted $\mathcal{T}_i$, and its endpoints will be labeled $A_i$ and $B_i$. Next we delete the vertex $Q$ with the two edges adjacent to it in $\Cl_1$, the resulting path has endpoints $P$ and $B$. Then we delete the vertex $P$ with the two edges adjacent to it in $\Cl_{\alpha}$, the resulting path has endpoints $A$ and $Q$. 

Now we want to connect $B$ with $A_2$, $B_2$ with $A_3$, $B_3$ with $A_4$, and so on, until connecting $B_{\alpha -1}$ with $A$. When this is done, we would finally connect $P$ and $Q$ using $\ell$ to get a cycle of the desired length and properties. However, we still need to prove that the lines needed to create this cycle, in this last part of the construction, have not been previously used.

Note that the line $\overleftrightarrow{BA_2}$ has not been used so far, as $B\in \Pi_1$ and $A_2\notin\Pi_1$ and all lines previously used connected points on the same $\Pi_i$. A similar argument applies to all other needed lines. However, it is possible that $\overleftrightarrow{B_2A_3} = \overleftrightarrow{BA_2}$. If this is the case then we can use Lemma \ref{lemPGn-1inPGn}\textbf{(b)}  to `move' $A_3$ away from the point $\overleftrightarrow{BA_2}\cap \Pi_2$. By labeling  this new point as $A_3$ we get that $\overleftrightarrow{B_2A_3} \neq \overleftrightarrow{BA_2}$. We perform this process of `moving' and labeling $\alpha - 2$ times, for all points $A_i$, $3\leq i \leq \alpha -1$, and for $A$. In this way we get the lines needed for this construction to be all distinct. Note that when `moving' $A_i$ we have already determined $i-2$ lines, and thus there are at most $i-2$ points that are forbidden for $A_i$ to move to. So, we would need the points to be moved to have at least $\alpha -2\leq q-2$ possible final locations (at least that many elements in its orbit), which is granted by Lemma \ref{lemPGn-1inPGn}\textbf{(b)}.\\

\noindent \textbf{Case 2. $\beta> 2$:} \\ 
This construction is almost the same as that in Case 1. The only difference is that, for all $1\leq i \leq \alpha$, we now consider cycles $\Cl_i$ in $\Pi_i$ of length $q^d+2$ with exactly two points of $\Sigma$ as vertices, and exactly one line of $\Sigma$ as edge, and also one cycle $\Cl_{\alpha+1}$ in $\Pi_{\alpha+1}$ of length $\beta$   with exactly two points of $\Sigma$ as vertices, and exactly one line of $\Sigma$ as edge. 
\end{proof}

Now we are ready to prove our main result.

\begin{proof}[Proof of Theorem 1]
We will prove this by induction on $n$. The case $n=2$ was proved in \cite{LMV13}. Hence, we assume the claim is true for $n\leq d$. \\
Now notice that Lemma \ref{lemshortedcycles} gives us $k$-cycles for most values of $k$. We only need to study the cases $q^{n}+3\leq k \leq q^{n}+q^{n-1}+\cdots + q+1$. 

For $n=d+1\geq 3$, we `partition' $PG(d+1,q)$ into $q+1$ copies of $PG(d, q)$, labelled $\Pi_1, \ldots , \Pi_{q+1}$, that share $\Sigma \cong PG(d-1,q)$. We also write $k=q^{d+1} +\beta$, where $3\leq \beta \leq q^d+q^{d-1}+\cdots +q+1$. Notice that the induction hypothesis yields $\beta$-cycles in $\Pi_{q+1}$, for all $3\leq \beta \leq q^d+q^{d-1}+\cdots +q+1$.

Let $\ell$ be a line of the $\beta$-cycle $\Cl_{q+1}$ in $\Pi_{q+1}$. WLOG, $\ell$ is a line of $\Sigma \cong PG(d-1,q)$. We use Lemma \ref{lemshortedcycles} to get a cycle $\Cl$ with length $q^{d+1} +2$ in $PG(d+1,q)$ with exactly two points of $\Pi_{q+1}$ as vertices, and exactly one line of $\Pi_{q+1}$ as edge. We next use Lemma \ref{lemPGn-1inPGn} to get the edge of $\Cl$ in $\Pi_{q+1}$ to be $\ell$. Finally, we delete the edge $\ell$ from $\Cl$ and from the $\beta$-cycle $\Cl_{q+1}$. The union of the resulting paths is a cycle with the desired length.
\end{proof}







\end{document}